\documentclass[table]{amsart}
\usepackage[margin=1.1in]{geometry}
\usepackage{amscd,amsmath,amsxtra,amsthm,amssymb,stmaryrd,xr,mathrsfs,mathtools,enumerate,commath, comment, mathtools}
\usepackage{tikz}
\usetikzlibrary{arrows.meta,decorations.markings,shapes.geometric,positioning}
\usetikzlibrary{calc}
\usepackage{tikz-3dplot}
\usepackage{stmaryrd}
\usepackage{multirow}
\usepackage{xcolor}
\usepackage{commath}
\usepackage{orcidlink}
\usepackage{comment}
\usepackage{svg}
\usepackage{graphics}
\usepackage{longtable} 
\usepackage{pdflscape} 
\usepackage{booktabs}
\usepackage{hyperref}
\definecolor{vegasgold}{rgb}{0.77, 0.7, 0.35}
\definecolor{darkgoldenrod}{rgb}{0.72, 0.53, 0.04}
\definecolor{gold(metallic)}{rgb}{0.83, 0.69, 0.22}
\hypersetup{
 colorlinks=true,
 linkcolor=darkgoldenrod,
 filecolor=brown,      
 urlcolor=gold(metallic),
 citecolor=darkgoldenrod,
 }
\newtheorem{lthm}{Theorem}

\usepackage[all,cmtip]{xy}

\DeclareFontFamily{U}{wncy}{}
\DeclareFontShape{U}{wncy}{m}{n}{<->wncyr10}{}
\DeclareSymbolFont{mcy}{U}{wncy}{m}{n}
\DeclareMathSymbol{\Sh}{\mathord}{mcy}{"58}
\usepackage[T2A,T1]{fontenc}
\usepackage[OT2,T1]{fontenc}

\newtheorem{theorem}{Theorem}[section]

\newtheorem*{theorem*}{Theorem}
\newtheorem*{ass*}{Assumption}
\newtheorem{definition}[theorem]{Definition}

\newtheorem{proposition}[theorem]{Proposition}

\newcommand{\Z}{\mathbb{Z}}
\newcommand{\Q}{\mathbb{Q}}

\newcommand{\op}[1]{\operatorname{#1}}

 \DeclareMathSymbol{\sha}{\mathord}{mcy}{"58}
 \makeatletter
\newcommand{\mylabel}[2]{#2\def\@currentlabel{#2}\label{#1}}
\makeatother

\numberwithin{equation}{section}

\title[Multivariable automatic arrays]{Multivariable automatic arrays and transcendence}

\author[A.~Paul]{Aadrita Paul\, \orcidlink{0009-0000-6958-3310}}
\address{Chennai Mathematical Institute, Chennai, India}
\email{aadritap.ug2023@cmi.ac.in}
\author[A.~Ray]{Anwesh Ray\, \orcidlink{0000-0001-6946-1559}}
\address{Chennai Mathematical Institute, Chennai, India}
\email{anwesh@cmi.ac.in}


\begin{document}

\maketitle

\begin{abstract}
We study real numbers defined by multidimensional automatic arrays weighted by multiplicatively independent bases. Let $a_1, \dots, a_r\geq 2$ be integers such that $\log a_1, \dots, \log a_r$ are $\Q$-linearly independent. Given bounded automatic sequences $(p_n(i))_{n\geq 0}$ with $i=1, \dots , r$ and a function $f:\Z^r\rightarrow \Z$, we consider the associated series
\[
\alpha = \sum_{n_1,\dots,n_r \geq 0} \frac{f(p_{n_1}(1),\dots,p_{n_r}(r))}{a_1^{n_1}\cdots a_r^{n_r}}.
\]
Using combinatorial properties of automatic sequences and Schmidt's Subspace Theorem, we prove that $\alpha$ is either rational or transcendental. This extends a result of Adamczewski and Bugeaud to the multidimensional setting.
\end{abstract}

\section{Introduction}
\subsection{Historical context and motivation}
\par The interaction between automata theory and Diophantine approximation has produced a number of striking results in transcendence theory. A classical way to measure the combinatorial complexity of an infinite word $\mathbf{u} = (u_n)_{n \geq 0}$ over a finite alphabet is via its subword complexity function $p(n)$, which counts the number of distinct blocks of length $n$ occurring in $\mathbf{u}$. This notion has proved to be a powerful bridge between combinatorics on words, theoretical computer science and Diophantine approximation. For instance, the $b$-ary expansion of a normal real number satisfies $p(n)=b^n$ for all $n$, reflecting maximal complexity. Ferenczi and Mauduit \cite{FerencziMauduit} established the transcendence of real numbers whose $b$-ary expansion has minimal complexity, namely $p(n)=n+1$ for all $n \geq 1$. Such sequences are precisely the \emph{Sturmian sequences} introduced by Morse and Hedlund \cite{MorseHedlund1, MorseHedlund2}. Their work relies on a refinement of a theorem of Ridout \cite{Ridout} and has subsequently inspired a number of results exhibiting transcendental numbers with low combinatorial complexity; see, for instance, \cite{AlloucheZamboni, Allouche,  RisleyZamboni, AdamczewskiCassaigne}.

\par A real number is said to be computable in time $T(n)$ if its first $n$ digits in some fixed base can be produced by a Turing machine in at most $T(n)$ steps. Of particular interest are numbers computable in linear time, also referred to as real-time computable numbers. Rational numbers clearly fall into this class, and Hartmanis and Stearns \cite{HartmanisStearns} raised the natural question of whether irrational algebraic numbers can also exhibit such low computational complexity. Loxton and van der Poorten \cite{LoxtonVanderPoorten1, LoxtonVanderPoorten2} proved that the base-$b$ expansion of an irrational algebraic number cannot be generated by a finite automaton. Although an initial proof based on Mahler’s method \cite{Mahler1, Mahler2, Mahler3} was proposed, it was later found to contain a substantial gap; see Becker \cite{Becker} and the discussion in \cite{Waldschmidt}. These developments culminated in the work of Adamczewski and Bugeaud \cite{AdamczewskiBugeaud}, who introduced powerful Diophantine methods, based on the Schmidt Subspace Theorem, to study the arithmetic nature of real numbers defined by low-complexity expansions.

\par Sequences generated by finite automata exhibit combinatorial properties which in turn lead to rigid arithmetic constraints on associated generating functions and infinite series. A fundamental illustration of this phenomenon is provided by Christol's theorem, which asserts that a power series $A(X)=\sum_{n\geq 0} a_n X^n$ with coefficients in a finite field of characteristic $p$ is algebraic over the field of rational functions if and only if the sequence $(a_n)_{n\geq 0}$ is $p$-automatic. In a different direction, Adamczewski and Bugeaud proved that if $b\geq 2$ is an integer and $(a_n)_{n\geq 0}$ is a bounded automatic sequence of integers, then the real number
\[
\alpha=\sum_{n\geq 0} \frac{a_n}{b^n}
\]
is either rational or transcendental; moreover, $\alpha$ is rational if and only if $(a_n)_{n\geq 0}$ is eventually periodic.
\subsection{Main results}
\par The present work may be viewed as a natural continuation of this line of inquiry, in which we establish a multidimensional analogue of the above result. To describe our setting, let $r \geq 2$ be an integer and let $a_1,\dots,a_r \geq 2$ be multiplicatively independent integers, that is, the only solution $(m_1,\dots,m_r)\in \mathbb{Z}^r$ to
\[
a_1^{m_1}\cdots a_r^{m_r}=1
\]
is $(m_1,\dots,m_r)=(0,\dots,0)$; equivalently, $\log a_1,\dots,\log a_r$ are linearly independent over $\mathbb{Q}$. In this higher-dimensional setting, we show that analogous rigidity phenomena persist for real numbers defined by multidimensional automatic data.

For each $1 \leq i \leq r$, let $(p_n(i))_{n\geq 0}$ be a bounded sequence with values in $\Z$, and assume that $(p_n(i))$ is $k_i$-automatic for some integer $k_i \geq 2$. Given a function
\[
f:\Z^r \to \Z,
\]
we associate to these data the $r$-dimensional array $\mathbf{c}=(c_{n_1,\dots,n_r})$ defined by
\[
c_{n_1,\dots,n_r} := f\bigl(p_{n_1}(1),\dots,p_{n_r}(r)\bigr),
\]
and consider the corresponding real number
\[
\alpha = \sum_{n_1,\dots,n_r \geq 0} \frac{c_{n_1,\dots,n_r}}{a_1^{n_1}\cdots a_r^{n_r}}.
\]
The array $\mathbf c$ is a special case of an $r$-dimensional automatic array in the sense of \cite[\S14.2]{alloucheshallit}. Our main result is the following.

\begin{lthm}\label{main thm}
Let $r \geq 2$, $a_1,\dots,a_r \geq 2$ and $\alpha$ be as above. Then $\alpha$ is either rational or transcendental.
\end{lthm}

The study of such numbers brings together methods from the theory of automatic sequences and higher-dimensional Diophantine approximation. A key feature of automatic sequences is that they satisfy a stammering condition: arbitrarily long prefixes contain large repeated blocks with controlled overlap (cf.\ \cite[proof of Theorem~2]{AdamczewskiBugeaud}). On the Diophantine side, the Schmidt Subspace Theorem provides a powerful framework for treating simultaneous approximation problems. For each $n$ we define an auxiliary rational number $\alpha_n$ by replacing each sequence $(p_k(i))_{k\geq 0}$ with an eventually periodic sequence. These rationals $\alpha_n$ admit explicit rational expressions with controlled denominators. The stammering property ensures that $\alpha_n$ approximates $\alpha$ with sufficient precision. These approximations are then shown, via a suitable application of the Subspace Theorem to preclude the possibility that $\alpha$ is algebraic and irrational.
\subsection{Outlook}
\par Our main result naturally motivates the following question: Let $a_1, \dots, a_r\geq 2$ be multiplicatively independent integers. Given integers $k_1, \dots, k_r \geq 2$ and any bounded $[k_1, \dots, k_r]$-automatic array $\mathbf{c}=(c_{n_1,\dots,n_r})$ of integers, can one still expect the associated series
\[
\sum_{n_1,\dots,n_r \geq 0} \frac{c_{n_1,\dots,n_r}}{a_1^{n_1}\cdots a_r^{n_r}}
\]
to be either rational or transcendental? We hope that the methods developed here will stimulate further investigation into multidimensional analogues of classical problems in transcendence theory.

\subsection*{Data availability} This manuscript has no associated data.
\subsection*{Conflict of interest statement} There is no conflict of interest to report.

\section{Automatic sequences and their properties}

\subsection{Words and morphisms}\label{s 2.1}
\par Let $\Sigma$ be a nonempty finite set, which we call an \emph{alphabet}. For an integer $k \geq 2$, set $\Sigma_k := \{0,1,\dots,k-1\}$, interpreted as the set of residue classes modulo $k$. A (finite) \emph{word} over $\Sigma$ is a sequence $w = w_0 \cdots w_{\ell}$ with $w_i \in \Sigma$. Its \emph{length} is given by $|w| := \ell + 1$. We write $\Sigma^*$ for the set of all finite words over $\Sigma$, including the empty word $\varepsilon$. The set $\Sigma^*$ is a monoid with respect to concatenation, i.e., for $x = x_1 \cdots x_n\in \Sigma^*$ and $y = y_1 \cdots y_m\in \Sigma^*$, set
\[
xy := x_1 \cdots x_n y_1 \cdots y_m.
\]
\noindent Given $x\in \Sigma^*$ and $z\in \Sigma^*$, we say that $x$ is a \emph{prefix} of $z$ if $z=xy$ for some word $y\in \Sigma^*$. 
\par The set of all \emph{infinite words} in the alphabet $\Sigma$ is denoted $\Sigma^\omega$ and consists of words of the form $w=w_0w_1\dots w_n\dots$. Given $x=x_1\dots x_n\in \Sigma^*$, denote by $x^\omega$ the infinite periodic word 
\[x^\omega=xxx\cdots.\]We may identify $\Sigma^\omega$ with the set of functions from $\Z_{\geq 0}$ to $\Sigma$, mapping $n\mapsto w_n$. We shall set $\Sigma^\infty:=\Sigma^*\cup \Sigma^\omega$. 
\par Let $\Sigma$ and $\Delta$ be alphabets. A function $\varphi: \Sigma^*\rightarrow \Delta^*$ is called a \emph{morphism} if it respects the monoidal structure, i.e., $\varphi(xy)=\varphi(x)\varphi(y)$ for all $x,y \in \Sigma^*$. Suppose that there exists $k\in \Z_{\geq 2}$ such that 
\[|\varphi(a)|=k\quad\text{ for all } a\in \Sigma.\] In this case, $\varphi$ is said to be \emph{k-uniform}. A $1$-uniform morphism is called a \emph{coding} and arises from a function $\varphi: \Sigma\rightarrow \Delta$.
\par Suppose that $\Sigma=\Delta$ and let $\varphi:\Sigma^*\rightarrow \Sigma^*$ be a $k$-uniform morphism. Note that $\varphi$ natural extends to a function $\varphi^\omega: \Sigma^\omega\rightarrow \Sigma^\omega$ defined by 
\[\varphi^\omega(w_0 w_1 \dots w_n \dots):=\varphi(w_0)\varphi(w_1)\dots \varphi(w_n)\dots.\]Let $a\in \Sigma$ be such that $\varphi(a)=ax$ for some $x\in \Sigma^*$ with $|x|=k-1$. Then $\varphi$ is said to be \emph{prolongable} on $a$. In this case the infinite word 
\[w:=\varphi^\omega(a):=ax \varphi(x)\varphi^2(x)\dots \varphi^n(x)\dots \] is fixed by $\varphi$. Indeed, one finds that 
\[\varphi(w)=\varphi\left(ax \varphi(x)\varphi^2(x)\dots \varphi^n(x)\dots\right)=w.\]
\subsection{Languages and DFAOs}
\par Given an alphabet $\Sigma$, a subset $L$ of $\Sigma^*$ is called a \emph{language}. 

\begin{definition}\label{DFA defn}A \emph{deterministic finite automaton} (DFA) over $\Sigma$ is a tuple
\[
\mathcal{A} = (Q,\Sigma,\delta,q_0,F),
\]
where $Q$ is a finite set (the set of states), $\delta : Q \times \Sigma \to Q$ is the transition function, $q_0 \in Q$ is the initial state, and $F \subseteq Q$ is the set of accepting states.
\end{definition}We can think of $\mathcal{A}$ as a finite directed graph for which the states are vertices and the edges are labelled by elements in $\Sigma$. This graph inputs finite words in $\Sigma^*$ and either accepts or declines them. In order to make this precise, we extend the transition function $\delta$ to a map $\delta : Q \times \Sigma^* \to Q$ according to the rules
\[
\delta(q,\varepsilon) = q, \qquad \delta(q,aw) = \delta(\delta(q,a),w)
\]
for $q \in Q$, $a \in \Sigma$, and $w \in \Sigma^*$. Then a word $w \in \Sigma^*$ is \emph{accepted} by $\mathcal{A}$ if $\delta(q_0,w) \in F$. The \emph{language recognized} by $\mathcal{A}$ is
\[
L(\mathcal{A}) := \{\, w \in \Sigma^* \mid \delta(q_0,w) \in F \,\}.
\]

\begin{definition}\label{regular language}A language $L \subseteq \Sigma^*$ is called \emph{regular} if there exists a DFA $\mathcal{A}$ such that $L = L(\mathcal{A})$.
\end{definition}
Given a language $L\subseteq \Sigma^*$, it is natural to ask when it arises from a DFA. Further one would like to investigate the structure of the DFA with the minimal number of states. This can be achieved via the famous Myhill-Nerode theorem, which we now recall. Define an equivalence relation $u \sim_L v$ for $u,v\in \Sigma^*$ if \[uw \in L \iff vw \in L \quad \forall w \in \Sigma^*.\]

\begin{theorem}[Myhill--Nerode]\label{myhill nerode}
With respect to notation above, the following are equivalent:
\begin{enumerate}
\item $L$ is regular.
\item The relation $\sim_L$ has finitely many equivalence classes.
\end{enumerate}
\end{theorem}
One can think of a DFA $\mathcal{A}$ as a machine which gives rise to a function $f_{\mathcal{A}}: \Sigma^*\rightarrow \{0, 1\}$, where $f(w)=1$ if $w$ is accepted and $f(w)=0$ if not. We shall consider functions $f: \Sigma^*\rightarrow \Delta$, where $\Delta$ is a finite alphabet which arise from DFAs with output. 
\begin{definition}
A \emph{DFAO} (deterministic finite automaton with output) is a $6$-tuple
\[
\mathcal{A} = (Q,\Sigma,\delta,q_0,\Delta,\tau),
\]
where $Q$ is a finite set of states, $\Sigma$ is the input alphabet, $\delta:Q\times \Sigma \to Q$ is the transition function, $q_0 \in Q$ is the initial state, $\Delta$ is a finite output alphabet, and $\tau:Q \to \Delta$ is the output function. If $\Sigma=\Sigma_k$, we call $\mathcal{A}$ a k-DFAO.
\end{definition}

As in the DFA case, the transition function $\delta$ extends uniquely to a map
\[
\delta: Q \times \Sigma^* \to Q
\]
according to the rules
\[
\delta(q,\varepsilon)=q, \qquad \delta(q,aw)=\delta(\delta(q,a),w),
\]
for $q \in Q$, $a \in \Sigma$, and $w \in \Sigma^*$. Consider the function $f_{\mathcal{A}}: \Sigma^*\rightarrow \Delta$ defined by 
\[f_{\mathcal{A}}(w):=\tau\left(\delta(q_0, w)\right).\] A function $f: \Sigma^*\rightarrow \Delta$ that arises from a DFAO $\mathcal{A}$ is called a \emph{finite state function}.

\par Next we recall the definition of an automatic sequences in a finite alphabet $\Delta$. Given a nonzero integer $n \geq 0$, we write \[n = \sum_{i=0}^r w_i k^i\] with $w_i \in \Sigma_k$ and $w_r \neq 0$. We associate to $n$ the word
\[
(n)_k := w_0 w_1 \cdots w_r \in \Sigma_k^*.
\] If $n=0$ we simply set $(0)_k:=0$. On the other hand, given a string $w_0\dots w_j\in \Sigma_k^*$, we set 
\begin{equation}\label{bracket k notation}[w_0\dots w_j]_k:=\sum_{i=0}^j w_i k^i.\end{equation}

\begin{definition}
A sequence $\mathbf{a}=(a_n)_{n \geq 0}$ with values in a finite alphabet $\Delta$ is called \emph{$k$-automatic} if there exists a $k$-DFAO $\mathcal{A}=(Q,\Sigma_k,\delta,q_0,\mathfrak{A},\tau)$ such that
\[
a_n = \tau\bigl(\delta(q_0,(n)_k)\bigr)
\]
for all $n \geq 0$.
\end{definition}
\noindent Equivalently, a sequence is $k$-automatic if its $n$-th term is obtained by feeding the base-$k$ expansion of $n$ into a finite automaton with output and reading the resulting state through the map $\tau$. For background and further properties, see \cite[\S4.3]{alloucheshallit}.
\par For a sequence $\mathbf{a}=(a_n)_{n \geq 0}$, define the $k$-kernel of $\mathbf{a}$ to be the set of subsequences \[K_k(\mathbf{a}):= \{(u_{k^in+j})_{n\geq 0}: i \geq 0\text{ and }0 \leq j \leq k^i\}.\]
The following result shows that the $k$-automaticity property is detected by the $k$-kernel.
\begin{theorem}
Let $k\geq2$. The sequence $\mathbf{a}=(a_n)_{n \geq 0}$ is $k$-automatic if and only if $K_k(\mathbf{a})$ is finite.
\end{theorem}
In the case where $\mathbf{a}$ takes values in $\{0,1\}$, this criterion follows from the Myhill-Nerode Theorem (i.e., Theorem \ref{myhill nerode}).

\par We recall the relationship between automatic sequences and sequences that are fixed under a morphism. 
\begin{theorem}[Cobham]\label{Cobham theorem on morphisms}
Let $\mathbf{a}=(a_n)_{n \geq 0}$ be a sequence with values in a finite alphabet $\Delta$. Then $\mathbf{a}$ is $k$-automatic if and only if there exist
\begin{itemize}
\item a finite alphabet $\mathfrak{B}$,
\item a $k$-uniform morphism $\sigma : \mathfrak{B}^* \to \mathfrak{B}^*$,
\item a letter $b \in \mathfrak{B}$ such that $\sigma(b)=bw$ for some $w \in \mathfrak{B}^*$ (i.e., $\varphi$ is prolongable in $b$ as in \S\ref{s 2.1}),
\item and a coding $\pi : \mathfrak{B} \to \Delta$,
\end{itemize}
such that $\mathbf{a} =\pi(\mathbf b)$ where $\mathbf b:=\sigma^\omega(b)$.
\end{theorem}

\subsection{Stammering sequences} 
Let $\Delta$ be a finite alphabet. For a word $W$ over $\Delta$, we denote by $|W|$ its length. For any integer $\ell \geq 1$, we write $W^\ell$ for $W\cdots W$, taken $\ell$-times. More generally, if $x>0$ is a real number, we define
\[
W^x := W^{\lfloor x \rfloor} W',
\]
where $W'$ is the prefix of $W$ of length $\lceil (x-\lfloor x \rfloor)|W| \rceil$. Here $\lfloor \cdot \rfloor$ and $\lceil \cdot \rceil$ denote, respectively, the floor and ceiling functions.

Let $\mathbf{a}=(a_k)_{k \geq 0}$ be a sequence with values in $\Delta$, which we identify with the infinite word $a_0 a_1 a_2 a_3 \cdots$.

\begin{definition}\label{stammering defn}
Let $w>1$ be a real number. We say that $\mathbf{a}$ satisfies \emph{Condition $(*)_w$} if there exist sequences of finite words $(U_n)_{n \geq 1}$ and $(V_n)_{n \geq 1}$ over $\Delta$ such that:
\begin{enumerate}
\item For every $n \geq 1$, the word $U_n V_n^{\,w}$ is a prefix of $\mathbf{a}$;
\item The sequence $\bigl(|U_n|/|V_n|\bigr)_{n \geq 1}$ is bounded;
\item The sequence $(|V_n|)_{n \geq 1}$ is strictly increasing.
\end{enumerate}
A sequence $\mathbf{a}$ satisfying Condition $(*)_w$ for some $w>1$ is called a \emph{stammering sequence}.
\end{definition}
\noindent It is easy to see that if $\mathbf{a}$ is eventually periodic then it satisfies $(*)_w$ for every $w>1$.
\begin{proposition}\label{automatic implies stammering}
Let $\mathbf{a}=(a_k)_{k \geq 0}$ be a $k$-automatic sequence on a finite alphabet $\mathcal{A}$, then $\mathbf{a}$ is a stammering sequence.
\end{proposition}

\begin{proof}
By Theorem \ref{Cobham theorem on morphisms} there exist a finite alphabet $\mathcal{B}$, a $k$-uniform morphism $\sigma:\mathcal{B}^* \to \mathcal{B}^*$, an infinite sequence $\mathbf{u}\in \mathcal{B}^\omega$ which is a fixed point of $\sigma$, and a coding $\phi:\mathcal{B}\to \mathcal{A}$ such that $\mathbf{a}=\phi(\mathbf{u})$. It is clear that if $\mathbf{u}$ satisfies Condition $(*)_w$, then so does $\mathbf{a}$. Thus it suffices to establish the result for $\mathbf{u}$.

Let $r=|\mathcal{B}|$. By the pigeonhole principle, the prefix of $\mathbf{u}$ of length $r+1$ contains two occurrences of the same letter. Hence it can be written in the form
\[
W_1\, u\, W_2\, u\, W_3,
\]
where $u \in \mathcal{B}$ and $W_1,W_2,W_3$ are (possibly empty) words over $\mathcal{B}$. For each integer $n \geq 1$, define
\[
U_n := \sigma^n(W_1), \qquad V_n := \sigma^n(uW_2).
\]
Because $V_n$ begins with $\sigma^n(u)$, we obtain that $U_n V_n^{\,1+1/r}$ is a prefix of $\mathbf{u}$.

We now verify the conditions of $(*)_w$ with $w=1+\frac{1}{r}$. Since $\sigma$ is $k$-uniform, we have $|\sigma^n(W)| = k^n |W|$ for any word $W$, and thus
\[
\frac{|U_n|}{|V_n|} = \frac{|W_1|}{|uW_2|} \leq \frac{|W_1|}{1+|W_2|} \leq r-1,
\]
so that the sequence $(|U_n|/|V_n|)$ is bounded. Moreover since $|V_n| = k^n |uW_2|$, it follows that $|V_n|$ is strictly increasing. Thus $\mathbf{u}$ satisfies Condition $(*)_{1+1/r}$, and therefore is a stammering sequence. The same holds for $\mathbf{a}$.
\end{proof}

\subsection{$r$-dimensional automatic arrays}
In this section we recall the notion of multidimensional automatic arrays. By way of example, first consider the $2$-dimensional case. Let $k,l \geq 2$ be integers, and let $\Sigma_k=\{0,1,\dots,k-1\}$ and $\Sigma_l=\{0,1,\dots,l-1\}$. We consider the product alphabet $\Sigma_k \times \Sigma_l$, consisting of pairs $[a,b]$ with $a\in \Sigma_k$ and $b\in \Sigma_l$. If
\[
w = [a_0,b_0][a_1,b_1][a_2,b_2]\cdots [a_j,b_j] \in (\Sigma_k \times \Sigma_l)^*,
\]
we define
\[
[w]_{k,l} := \bigl([a_0a_1\cdots a_j]_k,\; [b_0b_1\cdots b_j]_l\bigr), 
\]
where $[a_0a_1\cdots a_j]_k$ and $[b_0b_1\cdots b_j]_l$ are defined according to \eqref{bracket k notation}.
Conversely, given $(m,n) \in \mathbb{Z}_{\geq 0}^2$, write $(m)_k = a_0\cdots a_i$ and $(n)_l = b_0\cdots b_j$. We define a word $(m,n)_{k,l} \in (\Sigma_k \times \Sigma_l)^*$ by padding the shorter expansion with leading zeros, namely
\[
(m,n)_{k,l} =
\begin{cases}
[0,b_0]\cdots[0,b_{j-i}][a_0,b_{j-i+1}]\cdots[a_i,b_j], & \text{if } j \geq i,\\
[a_0,0]\cdots[a_{i-j},0][a_{i-j+1},b_0]\cdots[a_i,b_j], & \text{if } i>j.
\end{cases}
\]
\par More generally, let $r \geq 2$ be an integer, and let $k_1,\dots,k_r \geq 2$ be integers. For each $1 \leq i \leq r$, set $\Sigma_{k_i}=\{0,1,\dots,k_i-1\}$, and consider the product alphabet
\[
\Sigma := \Sigma_{k_1} \times \cdots \times \Sigma_{k_r}.
\]
Thus, elements of $\Sigma$ are $r$-tuples $[a^{(1)},\dots,a^{(r)}]$ with $a^{(i)} \in \Sigma_{k_i}$.

If
\[
w = [a^{(1)}_0,\dots,a^{(r)}_0]\cdots [a^{(1)}_j,\dots,a^{(r)}_j] \in \Sigma^*,
\]
we define
\begin{equation}\label{multi [w]}[w]_{k_1,\dots,k_r} := \bigl([a^{(1)}_0\cdots a^{(1)}_j]_{k_1},\dots,[a^{(r)}_0\cdots a^{(r)}_j]_{k_r}\bigr).
\end{equation}

Conversely, given $\mathbf{n}=(n_1,\dots,n_r) \in \mathbb{Z}_{\geq 0}^r$, write $(n_i)_{k_i} = a^{(i)}_1\cdots a^{(i)}_{\ell_i}$ for each $i$. Let $\ell = \max_i \ell_i$, and pad each expansion with leading zeros to obtain words of length $\ell$. This produces a word
\[
(\mathbf{n})_{k_1,\dots,k_r} \in \Sigma^*,
\]
defined by
\[
(\mathbf{n})_{k_1,\dots,k_r}
=
[a^{(1)}_0,\dots,a^{(r)}_0]\cdots [a^{(1)}_\ell,\dots,a^{(r)}_\ell],
\]
where each $(a^{(i)}_0,\dots,a^{(i)}_\ell)$ is the base-$k_i$ expansion of $n_i$ padded with leading zeros.

\begin{definition}
A \emph{$[k_1,\dots,k_r]$-DFAO} is a tuple
\[
M=(Q,\Sigma,\delta,q_0,\Delta,\tau),
\]
where $Q$ is a finite set of states, $\Sigma=\Sigma_{k_1}\times\cdots\times\Sigma_{k_r}$, $\delta:Q \times \Sigma \to Q$ is the transition function, $q_0 \in Q$ is the initial state, $\Delta$ is a finite output alphabet, and $\tau:Q \to \Delta$ is the output map.
\end{definition}

As before, $\delta$ extends to a map $\delta:Q \times \Sigma^* \to Q$.

\begin{definition}
A $[k_1,\dots,k_r]$-DFAO $M$ \emph{generates} an $r$-dimensional array
\[
\mathbf{c}=(c_{n_1,\dots,n_r})_{n_1,\dots,n_r \geq 0}
\]
over $\Delta$ if for all $\mathbf{n}=(n_1, \dots, n_r) \in \mathbb{Z}_{\geq 0}^r$, one has
\[
c_{n_1,\dots,n_r} = \tau\bigl(\delta(q_0,(\mathbf{n})_{k_1,\dots,k_r})\bigr).
\]
We say that $\mathbf{c}$ is \emph{$[k_1,\dots,k_r]$-automatic} if it is generated by such a DFAO. If $k_1=\cdots=k_r=k$, we simply say that $\mathbf{c}$ is \emph{$k$-automatic}.
\end{definition}

The following result generalizes \cite[Theorem 14.2.4]{alloucheshallit}.

\begin{theorem}\label{r-dim automatic}
For each $1 \leq i \leq r$, let $(a_n(i))_{n \geq 0}$ be a $k_i$-automatic sequence with values in a finite set $\Delta_i$. Let
\[
f:\Delta_1 \times \cdots \times \Delta_r \to \Delta
\]
be any map. Then the $r$-dimensional array $\mathbf{c}=(c_{n_1,\dots,n_r})$ defined by
\[
c_{n_1,\dots,n_r} := f\bigl(a_{n_1}(1),\dots,a_{n_r}(r)\bigr)
\]
is $[k_1,\dots,k_r]$-automatic.
\end{theorem}

\begin{proof}
For each $i$, let 
\[
M_i=(Q_i,\Sigma_{k_i},\delta_i,q_i,\Delta_i,\tau_i)
\]
be a DFAO generating $(a_n(i))_{n\geq 0}$. Define a DFAO
\[
M=(Q,\Sigma,\delta,q_0,\Delta,\tau)
\]
as follows. Set $Q=Q_1 \times \cdots \times Q_r$ and $q_0=(q_1,\dots,q_r)$. For
\[
\mathbf{q}=(q_1',\dots,q_r') \in Q, \quad [a^{(1)},\dots,a^{(r)}] \in \Sigma,
\]
define
\[
\delta(\mathbf{q},[a^{(1)},\dots,a^{(r)}])
:=
\bigl(\delta_1(q_1',a^{(1)}),\dots,\delta_r(q_r',a^{(r)})\bigr).
\]
Finally, define
\[
\tau(q_1',\dots,q_r') := f\bigl(\tau_1(q_1'),\dots,\tau_r(q_r')\bigr).
\]
It is easy to see that $M$ generates $\mathbf{c}$, and the result follows.
\end{proof}
\section{Diophantine tools}

In this section we collect several results from Diophantine approximation that will be used later. We begin with classical one-dimensional approximation results, and then introduce the language of absolute values and heights, ending with Schlickewei's refinement of Schmidt's Subspace Theorem.

\subsection{Dirichlet and Roth}

We first recall Dirichlet's theorem on Diophantine approximation.

\begin{theorem}[Dirichlet]
Let $\alpha \in \mathbb{R}$ and $Q \geq 1$. Then there exist integers $p,q$ with $1 \leq q \leq Q$ such that
\[
\left| \alpha - \frac{p}{q} \right| \leq \frac{1}{qQ}.
\]
In particular, there exist infinitely many rational numbers $p/q$ such that
\[
\left| \alpha - \frac{p}{q} \right| < \frac{1}{q^2}.
\]
\end{theorem}

Dirichlet's theorem shows that every real number admits very good rational approximations. For algebraic numbers, however, such approximations cannot be substantially improved.

\begin{theorem}[Roth]
Let $\alpha$ be an irrational algebraic number. Then for every $\varepsilon > 0$, the inequality
\[
\left| \alpha - \frac{p}{q} \right| < \frac{1}{q^{2+\varepsilon}}
\]
has only finitely many solutions in rational numbers $p/q$.
\end{theorem}

Roth's theorem is a far-reaching refinement of earlier results of Thue and Siegel. A higher-dimensional and powerful generalization is the Subspace Theorem, which we now describe.

\subsection{Absolute values and heights}

Let $\mathbf{K}$ be an algebraic number field of degree $d=[\mathbf{K}:\mathbb{Q}]$, and let $M(\mathbf{K})$ denote the set of (normalized) places of $\mathbf{K}$. For $v \in M(\mathbf{K})$ and $x \in \mathbf{K}$, we define an absolute value $|\cdot|_v$ as follows:
\begin{enumerate}
\item[(i)] If $v$ corresponds to a real embedding $\sigma:\mathbf{K}\hookrightarrow \mathbb{R}$, set
\[
|x|_v := |\sigma(x)|^{1/d}.
\]
\item[(ii)] If $v$ corresponds to a pair of complex embeddings $\sigma,\overline{\sigma}:\mathbf{K}\hookrightarrow \mathbb{C}$, set
\[
|x|_v := |\sigma(x)|^{2/d} = |\overline{\sigma}(x)|^{2/d}.
\]
\item[(iii)] If $v$ corresponds to a nonzero prime ideal $\mathfrak{p} \subset \mathcal{O}_{\mathbf{K}}$, set
\[
|x|_v := (N\mathfrak{p})^{-\operatorname{ord}_{\mathfrak{p}}(x)/d}.
\]
\end{enumerate}

With this normalization, the family $\{|\cdot|_v\}_{v \in M(\mathbf{K})}$ satisfies the \emph{product formula}
\[
\prod_{v \in M(\mathbf{K})} |x|_v = 1 \qquad \text{for all } x \in \mathbf{K}^\times.
\]

For $\mathbf{x}=(x_1,\dots,x_n) \in \mathbf{K}^n \setminus \{\mathbf{0}\}$ and $v \in M(\mathbf{K})$, define
\[
|\mathbf{x}|_v :=
\begin{cases}
\left( \sum_{i=1}^n |x_i|_v^{2d} \right)^{1/(2d)}, & \text{if } v \text{ is real infinite},\\[0.5em]
\left( \sum_{i=1}^n |x_i|_v^{d} \right)^{1/d}, & \text{if } v \text{ is complex infinite},\\[0.5em]
\max\{|x_1|_v,\dots,|x_n|_v\}, & \text{if } v \text{ is finite}.
\end{cases}
\]

\begin{definition}
The (multiplicative) \emph{height} of $\mathbf{x} \in \mathbf{K}^n \setminus \{\mathbf{0}\}$ is defined by
\[
H(\mathbf{x}) := \prod_{v \in M(\mathbf{K})} |\mathbf{x}|_v.
\]
\end{definition}

This definition is independent of the choice of coordinates up to multiplication by a bounded factor, and plays a central role in Diophantine geometry. For each place $v$ of, fix an extension $|\cdot |_v$ to $\bar{\Q}$.

\subsection{The Subspace Theorem}

We now state the Subspace Theorem, due to Schmidt and its refinement by Schlickewei, which can be viewed as a multidimensional generalization of Roth's theorem.

\begin{theorem}[Subspace Theorem]\label{thm subspace}
Let $\mathbf{K}$ be an algebraic number field, and let $m \geq 2$ be an integer. Let $S \subset M(\mathbf{K})$ be a finite set of places containing all infinite places. For each $v \in S$, let
\[
L_{1,v},\dots,L_{m,v}
\]
be linearly independent linear forms in $m$ variables with algebraic coefficients. Let $\varepsilon>0$.

Then the set of $\mathbf{x} \in \mathbf{K}^m \setminus \{\mathbf{0}\}$ satisfying
\[
\prod_{v \in S} \prod_{i=1}^m \frac{|L_{i,v}(\mathbf{x})|_v}{|\mathbf{x}|_v}
\leq H(\mathbf{x})^{-m-\varepsilon}
\]
is contained in finitely many proper linear subspaces of $\mathbf{K}^m$.
\end{theorem}

\section{Multivariable arrays generated by $r$ automatic sequences}

\par In this section we prove Theorem \ref{main thm}. Fix $r\in \Z_{\geq 2}$ and let $a_1,a_2,\dots,a_r \in \Z_{\geq 2}$ be distinct integers which are multiplicatively independent, i.e., $\log a_1,\dots,\log a_r$ are linearly independent over $\Q$.

For each $1 \leq i \leq r$, let $(p_n(i))_{n\geq 0}$ be a bounded sequence with values in $\Z$, and assume that $(p_n(i))$ is $k_i$-automatic for some integer $k_i \in \Z_{\geq 2}$. Let
\[
f:\Z^r \to \Z
\]
be a function, and define an $r$-dimensional array $\mathbf{c}=(c_{n_1,\dots,n_r})$ by
\[
c_{n_1,\dots,n_r} := f\bigl(p_{n_1}(1),\dots,p_{n_r}(r)\bigr).
\]
We note that each of the sequences $(p_n(i))$ is bounded and hence so are the values $\{c_{n_1,\dots,n_r}\mid \mathbf n=(n_1, \dots, n_r)\in \Z_{\geq 0}^r\}$. It follows from Theorem~\ref{r-dim automatic} that $\mathbf{c}$ is a $[k_1,\dots,k_r]$-automatic array with values in a bounded subset $\Delta \subset \Z$.
\par Since each of the sequences $(p_n(i))_{n\geq 0}$, $1 \leq i \leq r$, is automatic, it follows from Proposition \ref{automatic implies stammering} that each satisfies Condition $(*)_{w_i}$ for some real number $w_i>1$. We set $w$ to be minimum value among $w_1, \dots, w_r$. Thus, for each $1 \leq i \leq r$, there exist sequences of finite words $(U_n(i))_{n\geq 1}$ and $(V_n(i))_{n\geq 1}$ such that:
\begin{itemize}
    \item for every $n\geq 1$, the word $U_n(i)\, V_n(i)^{\,w}$ is a prefix of the infinite word $p_0(i)p_1(i)p_2(i)\cdots$;
    \item the sequence $\bigl(|U_n(i)|/|V_n(i)|\bigr)_{n\geq 1}$ is bounded;
    \item the sequence $(|V_n(i)|)_{n\geq 1}$ is strictly increasing.
\end{itemize}
\noindent For $1 \leq i \leq r$ and $n\geq 1$, we set
\[
r_n(i):=|U_n(i)|\quad \text{and}\quad s_n(i):=|V_n(i)|.
\]
For a fixed $n\geq 1$, we construct auxiliary sequences $(p_k^{(n)}(i))_{k\geq 0}$ by extending the prefix $U_n(i)V_n(i)$ periodically. More precisely, for each $1 \leq i \leq r$, define
\[
p_k^{(n)}(i) :=
\begin{cases}
p_k(i), & 0 \leq k < r_n(i)+s_n(i),\\
p_{\,r_n(i)+h}(i), & k = r_n(i)+h + j s_n(i),\;\; 0 \leq h < s_n(i),\;\; j \geq 1.
\end{cases}
\]
\noindent Note that each sequence $(p_k^{(n)}(i))$ is eventually periodic, with preperiod $r_n(i)$ and period dividing $s_n(i)$.

\par Define an $r$-dimensional array
\[
c^{(n)}_{n_1,\dots,n_r} := f\bigl(p^{(n)}_{n_1}(1),\dots,p^{(n)}_{n_r}(r)\bigr),
\]
and set
\[
\alpha_n := \sum_{n_1,\dots,n_r \geq 0} \frac{c^{(n)}_{n_1,\dots,n_r}}{a_1^{n_1}\cdots a_r^{n_r}}.
\]

\medskip

\noindent We now express $\alpha_n$ in terms of geometric series. Using the eventual periodicity of each sequence $(p_k^{(n)}(i))$, it follows that $\alpha_n$ is a rational number. More precisely, one checks that $\alpha_n$ can be written as a fraction whose denominator is
\[
\prod_{i=1}^r a_i^{r_n(i)}\bigl(a_i^{s_n(i)}-1\bigr),
\]
and whose numerator is $P_n(a_1,\dots,a_r)$ for some polynomial \[P_n(X_1,\dots,X_r) \in \Z[X_1,\dots,X_r]
\] whose coefficients lie in a bounded set independent of $n$.

\begin{proposition}\label{repn of alphan}
With notation as above, for each $n\geq 1$ there exists a polynomial 
\[
P_n(X_1,\dots,X_r)\in \mathbb{Z}[X_1,\dots,X_r]
\]
such that
\[
\alpha_n
=
\frac{P_n(a_1,\dots,a_r)}
{\displaystyle \prod_{i=1}^r a_i^{r_n(i)}\bigl(a_i^{s_n(i)}-1\bigr)}.
\]
Moreover, for each $1 \leq i \leq r$, the degree of $P_n$ in $X_i$ is bounded by $r_n(i)+s_n(i)$, and its coefficients belong to a fixed finite subset of $\mathbb{Z}$, depending only on $f$ and the sequences $(p_n(i))$, but independent of $n$.
\end{proposition}

\begin{proof}
\par In order to better illustrate the argument, we first prove the result when $r=2$. In order to simplify notation, we use $(m, n)$ in place of $(n_1, n_2)$ and $(a,b)$ in place of $(a_1, a_2)$. Also set $p_n:=p_n(1)$ and $q_n:=p_n(2)$. Thus, $\alpha$ is given by 
\[\alpha=\sum_{m, n\geq 0}\frac{c_{m, n}}{a^m b^n}.\]Note that $c_{m, n}=f(p_m, q_n)$. Let $M>0$ be such that $c_{m,n}\in [-M, M]$ for all $m$ and $n$. Write
\[
\alpha_n = S_1 + S_2 + S_3 + S_4,
\]
where the sums are over points $\Z^2$ in the regions
\[
[0,r_n)\times[0,r'_n), \quad [0,r_n)\times[r'_n,\infty), \quad
[r_n,\infty)\times[0,r'_n), \quad [r_n,\infty)\times[r'_n,\infty)
\]
respectively.
\par The term
\[
S_1 = \sum_{0\leq i<r_n}\sum_{0\leq j<r'_n} \frac{c_{i,j}^{(n)}}{a^i b^j}
\]
is a finite sum, hence can be written as
\[
S_1 = \frac{P_{n,1}(a,b)}{a^{r_n} b^{r'_n}},
\]
where $P_{n,1}(X,Y)\in \mathbb{Z}[X,Y]$ has $X$-degree $\leq r_n$ and $Y$-degree $\leq r_n'$ with coefficients lying in $[-M, M]$. For $S_2$, using periodicity in the $j$-direction, we obtain
\[
S_2 = \sum_{0\leq i<r_n} \sum_{t=0}^{s'_n-1}
\frac{f(p_i^{(n)},q_{r'_n+t}^{(n)})}{a^i b^{r'_n+t}} \cdot \frac{1}{1-b^{-s'_n}}.
\]
Thus
\[
S_2 = \frac{P_{n,2}(a,b)}{a^{r_n} b^{r'_n}(b^{s'_n}-1)},
\]
for some polynomial $P_{n,2}\in \mathbb{Z}[X,Y]$ with $X$-degree $\leq r_n$ and $Y$-degree $\leq s_n'$. Similarly,
\[
S_3 = \frac{P_{n,3}(a,b)}{a^{r_n}(a^{s_n}-1) b^{r'_n}},
\]
for some $P_{n,3}\in \mathbb{Z}[X,Y]$ with $X$-degree $\leq s_n$ and $Y$-degree $\leq r_n'$. Note that by construction, $P_{n,2}$ and $P_{n, 3}$ have coefficients in $\Z$ that are of the form $f(p_m, q_n)$ and hence lie in $[-M, M]$. Using periodicity in both directions, we have
\begin{align*}
S_4
&= \sum_{t=0}^{s_n-1}\sum_{u=0}^{s'_n-1}
\frac{f(p_{r_n+t}^{(n)},q_{r'_n+u}^{(n)})}{a^{r_n+t} b^{r'_n+u}}
\cdot \frac{1}{(1-a^{-s_n})(1-b^{-s'_n})} \\
&= \frac{P_{n,4}(a,b)}{a^{r_n}(a^{s_n}-1)\, b^{r'_n}(b^{s'_n}-1)},
\end{align*}
for some polynomial $P_{n,4}\in \mathbb{Z}[X,Y]$ with $X$-degree $\leq s_n$ and $Y$-degree $\leq s_n'$.
\par Putting everything over the common denominator
\[
a^{r_n}(a^{s_n}-1)\, b^{r'_n}(b^{s'_n}-1),
\]
we obtain
\[
\alpha_n = \frac{P_n(a,b)}{a^{r_n}(a^{s_n}-1)\, b^{r'_n}(b^{s'_n}-1)},
\]
where
\[
P_n(X,Y) = P_{n,1}(X,Y)(X^{s_n}-1)(Y^{s'_n}-1)
+ P_{n,2}(X,Y)(X^{s_n}-1)
+ P_{n,3}(X,Y)(Y^{s'_n}-1)
+ P_{n,4}(X,Y).
\]
\noindent It is clear from the construction that $P_n(X,Y)\in \mathbb{Z}[X,Y]$, that its $X$-degree is $\leq r_n+s_n$ and $Y$-degree $\leq r_n'+s_n'$ respectively. Further, the coefficients of $P_n(X,Y)$ lie in the set $[-4M, 4M]$. 

\par Now let's complete the argument for all values of $r$. Let $M>0$ be such that
\[
c_{n_1,\dots,n_r} = f\bigl(p_{n_1}(1),\dots,p_{n_r}(r)\bigr)\in [-M,M]
\]
for all $(n_1,\dots,n_r)\in \Z_{\geq 0}^r$.

\par We decompose $\alpha_n$ according to whether each coordinate lies before or after the preperiod. For each subset $I \subseteq \{1,\dots,r\}$, define
\[
S_I := \sum_{\substack{n_i < r_n(i)\ \text{if } i\notin I \\ n_i \geq r_n(i)\ \text{if } i\in I}}
\frac{c^{(n)}_{n_1,\dots,n_r}}{a_1^{n_1}\cdots a_r^{n_r}}.
\]

Then
\[
\alpha_n = \sum_{I \subseteq \{1,\dots,r\}} S_I.
\]

\par Fix $I \subseteq \{1,\dots,r\}$. For $i \in I$, we use the periodicity of $(p_k^{(n)}(i))$ beyond $r_n(i)$, while for $i \notin I$ we sum over a finite range. Writing $n_i = r_n(i)+h_i + j_i s_n(i)$ with $0 \leq h_i < s_n(i)$ and $j_i \geq 0$ for $i\in I$, we obtain
\[S_I
=
\sum_{\substack{0 \leq n_i < r_n(i)\ (i\notin I)}}\;
\sum_{\substack{0 \leq h_i < s_n(i)\ (i\in I)}}
\frac{c^{(n)}_{\tilde{n}_1,\dots,\tilde{n}_r}}{\prod_{i=1}^r a_i^{\tilde n_i}}
\cdot
\prod_{i\in I} \left(\sum_{j_i\geq 0} a_i^{-j_i s_n(i)}\right),\]
where $\tilde n_i = n_i$ if $i\notin I$ and $\tilde n_i = r_n(i)+h_i$ if $i\in I$.

\par Evaluating the geometric series, we obtain
\[
\sum_{j_i\geq 0} a_i^{-j_i s_n(i)} = \frac{1}{1-a_i^{-s_n(i)}},
\]
and hence
\[
S_I
=
\frac{P_{n,I}(a_1,\dots,a_r)}
{\displaystyle \prod_{i\notin I} a_i^{r_n(i)} \prod_{i\in I} a_i^{r_n(i)}\bigl(a_i^{s_n(i)}-1\bigr)},
\]
for some polynomial $P_{n,I}\in \Z[X_1,\dots,X_r]$ with coefficients in $[-M, M]$ and for which 
\[\op{deg}_{X_i}P_{n, I}\leq \begin{cases}
    s_n(i)\text{ for }i\in I;\\
    r_n(i)\text{ for }i\notin I.
\end{cases}\]

\par Summing over all subsets $I \subseteq \{1,\dots,r\}$ and putting everything over the common denominator
\[
\prod_{i=1}^r a_i^{r_n(i)}\bigl(a_i^{s_n(i)}-1\bigr),
\]
we obtain
\[
\alpha_n
=
\frac{P_n(a_1,\dots,a_r)}
{\displaystyle \prod_{i=1}^r a_i^{r_n(i)}\bigl(a_i^{s_n(i)}-1\bigr)},
\]
where
\[
P_n(X_1,\dots,X_r)
=
\sum_{I \subseteq \{1,\dots,r\}} 
P_{n,I}(X_1,\dots,X_r)
\prod_{i\notin I} (X_i^{s_n(i)}-1).
\]

\medskip

It is clear from the construction that $P_n \in \Z[X_1,\dots,X_r]$. Moreover, for each $i$, the $X_i$-degree is bounded by $r_n(i)+s_n(i)$. Finally, the coefficients of all polynomials $P_{n,I}$ belong to $[-M,  M]$, Hence the coefficients of $P_n$ belong to $[-2^r M, 2^r M]$.
\end{proof}

\begin{proof}[Proof of Theorem \ref{main thm}]
    Assume that $\alpha\in \bar{\Q}$ then we show that $\alpha\in \Q$.
    
    \par In order to better illustrate the method, we first complete the proof when $r=2$. In keeping with the notation in the proof of the previous result, we use $(m, n)$ in place of $(n_1, n_2)$ and $(a,b)$ in place of $(a_1, a_2)$. Let $S_0$ be the set of all finite primes dividing $ab$. We identify $S_0$ with nonarchimedian places of $\Q$ and set $S:=S_0\cup \{\infty\}$. Let $\mathbf{x}:=(x_1, \dots, x_5)$. We define $5$ linear forms $L_{v,i}$ for $i=1, \dots, 5$ for each $v\in S$. If $v\in S_0$, we set $L_{v,i}(\mathbf x):=x_i$. On the other hand if $v=\infty$ we set 
    \[L_{\infty, 5}(\mathbf{x}):=\alpha(x_1+x_2+x_3+x_4)+x_5\] and 
    \[L_{\infty,i}(\mathbf x):=x_i\]for $i=1, \dots, 4$. We then take 
    \[\mathbf x_n:=\left(a^{r_n+s_n}b^{r_n'+s_n'}, -a^{r_n}b^{r_n'+s_n'}, -a^{r_n+s_n}b^{r_n'}, a^{r_n} b^{r_n'}, -P_n(a,b)\right).\]
    From Proposition \ref{repn of alphan}, we find that 
    \[\begin{split}L_{\infty, 5}(\mathbf x_n)=&\left(\alpha a^{r_n}(a^{s_n}-1)\, b^{r'_n}(b^{s'_n}-1)-P_n(a,b)\right)\\
    =& a^{r_n}(a^{s_n}-1)\, b^{r'_n}(b^{s'_n}-1)\left(\alpha-\alpha_n\right)\\
    =& a^{r_n}(a^{s_n}-1)\, b^{r'_n}(b^{s'_n}-1) \sum_{\substack{i\geq r_n+\lceil w s_n \rceil +1\\ j\geq r_n'+\lceil w s_n' \rceil +1}} \frac{c_{i,j}-c_{i,j}^{(n)}}{a^i b^j}\\
    =& O\left(\frac{1}{a^{(w-1)s_n} b^{(w-1)s_n'}}\right).
    \end{split}\]
    We wish to estimate the product
$$
\Pi:=\prod_{v \in S} \prod_{i=1}^{5} \frac{\left|L_{v,i}\left(\mathbf{x}_{n}\right)\right|_{v}}{\left|\mathbf{x}_{n}\right|_{v}}=\prod_{v\in S} \prod_{i=1}^4 |L_{v,i}(\mathbf{x}_n)|_v\times \prod_{v\in S} \frac{\left|L_{v,5}\left(\mathbf{x}_{n}\right)\right|_v}{\left|\mathbf{x}_{n}\right|_{v}^5}.
$$
Note that for $v\notin S$ and $i\leq 4$ we have that $|L_{v,i}(\mathbf{x}_n)|_v=1$.
Thus by the product formula, 
\[\prod_{v\in S} \prod_{i=1}^4 |L_{v,i}(\mathbf{x}_n)|_v=1.\]
Since the polynomial $P_n(X,Y)$ has integer coefficients it follows that $|L_{v,5}(\mathbf{x}_n)|_v\leq 1$ for any $v\in S_0$. Thus we find that 
\[\Pi\ll a^{-(w-1)s_n} b^{-(w-1)s_n'}\prod_{v\in S} |\mathbf x_n|_v^{-5}=a^{-(w-1)s_n} b^{-(w-1)s_n'}H(\mathbf x_n)^{-5}.\]
\noindent Note that for each nonarchimedian place $v$, $|\mathbf x_n|_v\leq 1$ and therefore, 
\[H(\mathbf x_n)\leq |\mathbf x_n|_\infty.\]
On the other hand, since $P_n(X,Y)$ has absolutely bounded coefficients and degree at most $r_n+s_n$ in $X$ and $r_n'+s_n'$ in $Y$, it follows that 
\[|\mathbf x_n|_\infty\ll (r_n+s_n)(r_n'+s_n') a^{r_n+s_n}b^{r_n'+s_n'}.\] Let $C>0$ be an absolute constant such that $r_n<C s_n$ and $r_n'<C s_n'$. We find that 
\[H(\mathbf x_n)\ll (r_n+s_n)(r_n'+s_n') (a^{s_n}b^{s_n'})^{C+1}\]
i.e., 
\[a^{-s_n}b^{-s_n'}\ll (r_n+s_n)^{1/(C+1)}(r_n'+s_n')^{1/(C+1)} H(\mathbf x_n)^{-1/(C+1)}\]
Therefore we find that
\[a^{-(w-1)s_n} b^{-(w-1)s_n'}\ll \left((r_n+s_n)(r_n'+s_n') H(\mathbf x_n)^{-1}\right)^{\frac{w-1}{C+1}}.\]
Note that \[H(\mathbf x_n)\gg a^{r_n+s_n} b^{r_n'+s_n'}\] and therefore, for any constant $C'>0$, we have that 
\[H(\mathbf x_n)\gg \left((r_n+s_n)(r_n'+s_n')\right)^{C'}.\]
Therefore, we find that 
\[(r_n+s_n)(r_n'+s_n')\ll H(\mathbf x_n)^{\delta}
\] for any $\delta>0$. 
Thus for any
$\epsilon\in (0, \frac{w-1}{C+1})$, we have that 
\[\left((r_n+s_n)(r_n'+s_n') H(\mathbf x_n)^{-1}\right)^{\frac{w-1}{C+1}}\ll H(\mathbf x_n)^{-\epsilon}. \]
We deduce that \[\Pi\ll H(\mathbf x_n )^{-5-\epsilon}\] for all $n$. Consequently infinitely many of the $\mathbf x_n$ lie on the same hyperplane. Passing to a subsequence, assume that all $\mathbf x_n$ satisfy 
\[(u_1, u_2, u_3, u_4, u_5)\cdot \mathbf x_n=0\] where the vector $\mathbf u=(u_1, u_2, u_3, u_4, u_5)\in \mathbb Q^5$ is not zero. 
Thus we find that 
\[u_1 a^{r_n+s_n}b^{r_n'+s_n'}-u_2a^{r_n}b^{r_n'+s_n'} -u_3 a^{r_n+s_n}b^{r_n'}+u_4 a^{r_n} b^{r_n'}-u_5P_n(a,b)=0.\]
Suppose that $u_5\neq 0$. Note that 
\[\lim_{n\rightarrow \infty} \frac{P_n(a,b)}{a^{r_n+s_n} b^{r'_n+s'_n}}=\lim_{n\rightarrow \infty} \frac{P_n(a,b)}{a^{r_n}(a^{s_n}-1)\, b^{r'_n}(b^{s'_n}-1)}=\lim_{n\rightarrow \infty} \alpha_n=\alpha.\]
Thus dividing by $a^{r_n+s_n} b^{r'_n+s'_n}$ and letting $n\rightarrow \infty$ we find that 
\[u_5 \alpha= u_1\] and so $\alpha\in \Q$. Next suppose that $u_5=0$. Then the same argument gives $u_1=u_5\alpha=0$. Thus we find that for all $n$, 
\[u_2b^{s_n'} +u_3 a^{s_n}=u_4 .\] Note that both $u_2$ and $u_3$ must be non-zero, otherwise it is easy to conclude that all $u_i=0$. If $u_4\neq 0$ then divide by $u_4$ to get 
\[ A x+By=1\] where $A, B$ are fixed rational numbers and $x:=b^{s_n'} $, $y:=a^{s_n}$ belong to finitely generated subgroup of $\Q^\times$. It then follows that there are only finitely many solutions to this equation, see \cite[Theorem 3.1, Ch.8]{lang}. Thus we find that $u_4=0$. In this case, 
\[u_2 b^{s_n'}=-u_3 a^{s_n}.\] In particular, 
\[(s_n'-s_m') \log b-(s_n-s_m) \log a=0,\] which is a contradiction. Thus we have shown that $\alpha\in \Q$.    

\par Next consider the $r$-dimensional case. Let $S_0$ be the set of all finite primes dividing $a_1\cdots a_r$, and set $S:=S_0\cup\{\infty\}$. We consider the vector
\[
\mathbf{x}:=(x_1,\dots,x_{2^r},x_{2^r+1}) \in \Q^{2^r+1},
\]
whose first $2^r$ coordinates correspond to the $2^r$ terms arising from the decomposition over subsets $I\subseteq \{1,\dots,r\}$, and whose last coordinate will specialize to $-P_n(a_1,\dots,a_r)$.

\par For $v\in S_0$ and $1\leq i\leq 2^r+1$, define
\[
L_{v,i}(\mathbf{x}) := x_i.
\]
For $v=\infty$, define
\[
L_{\infty,2^r+1}(\mathbf{x})
:=
\alpha\Bigl(\sum_{i=1}^{2^r} x_i\Bigr) + x_{2^r+1},
\qquad
L_{\infty,i}(\mathbf{x}) := x_i \quad (1\leq i\leq 2^r).
\]
\noindent Order the collection of subsets $I\subset \{1, \dots, r\}$ by $I_1, \dots, I_{2^r}$. For each $n\geq 1$, define $\mathbf{x}_n$ as follows. For each subset $I\subseteq \{1,\dots,r\}$, set
\[
x_I :=
(-1)^{|I|}\prod_{i\notin I} a_i^{r_n(i)+s_n(i)} \prod_{i\in I} a_i^{r_n(i)},
\]
and define
\[
\mathbf{x}_n := \bigl( (x_I)_{I\subseteq \{1,\dots,r\}},\; -P_n(a_1,\dots,a_r)\bigr).
\]
In other words, for $i\geq 2^r$ set $x_i:=x_{I_i}$ and $x_{2^r+1}:=-P_n(a_1, \dots, a_r)$.
\par By Proposition \ref{repn of alphan} we find that
\[
L_{\infty,2^r+1}(\mathbf{x}_n)
=
\left(\alpha \prod_{i=1}^r a_i^{r_n(i)}(a_i^{s_n(i)}-1) - P_n(a_1,\dots,a_r)\right)
=
\prod_{i=1}^r a_i^{r_n(i)}(a_i^{s_n(i)}-1)\,(\alpha-\alpha_n).
\]
\noindent Arguing as in the two-dimensional case, we have
\[
\alpha-\alpha_n
=
\sum_{\substack{n_i \geq r_n(i)+\lceil w s_n(i)\rceil+1 \\ 1\leq i\leq r}}
\frac{c_{n_1,\dots,n_r}-c^{(n)}_{n_1,\dots,n_r}}{a_1^{n_1}\cdots a_r^{n_r}},
\]
and hence
\[
L_{\infty,2^r+1}(\mathbf{x}_n)
=
O\!\left(\prod_{i=1}^r a_i^{-(w-1)s_n(i)}\right).
\]

\par We now estimate
\[
\Pi
:=
\prod_{v\in S}\prod_{i=1}^{2^r+1}
\frac{|L_{v,i}(\mathbf{x}_n)|_v}{|\mathbf{x}_n|_v}.
\]
As before, for $i\leq 2^r$ and $v\notin S$, we have $|L_{v,i}(\mathbf{x}_n)|_v=1$, and hence by the product formula,
\[
\prod_{v\in S}\prod_{i=1}^{2^r} |L_{v,i}(\mathbf{x}_n)|_v = 1.
\]
Moreover, since $P_n$ has integer coefficients, we have $|L_{v,2^r+1}(\mathbf{x}_n)|_v \leq 1$ for $v\in S_0$. Thus
\[
\Pi \ll \prod_{i=1}^r a_i^{-(w-1)s_n(i)} \cdot H(\mathbf{x}_n)^{-(2^r+1)}.
\]

\par On the other hand, since $P_n$ has bounded coefficients and degree $\ll r_n(i)+s_n(i)$ in each variable, we obtain
\[
H(\mathbf{x}_n)
\ll
\Bigl(\prod_{i=1}^r (r_n(i)+s_n(i))\Bigr)
\prod_{i=1}^r a_i^{r_n(i)+s_n(i)}.
\]
Using $r_n(i)\ll s_n(i)$, we deduce
\[
H(\mathbf{x}_n)
\ll
\Bigl(\prod_{i=1}^r (r_n(i)+s_n(i))\Bigr)
\prod_{i=1}^r a_i^{(C+1)s_n(i)}
\]
for some constant $C>0$.

\par It follows that there exists $\varepsilon>0$ such that
\[
\prod_{i=1}^r a_i^{-(w-1)s_n(i)}
\ll H(\mathbf{x}_n)^{-\varepsilon},
\]
and hence for all $n$,
\[
\Pi \ll H(\mathbf{x}_n)^{-(2^r+1+\varepsilon)}.
\]
\noindent Theorem \ref{thm subspace} then implies that infinitely many of the $\mathbf{x}_n$ lie in a proper subspace. Passing to a subsequence, we may assume that
\[
\sum_{i=1}^{2^r} u_i x_{I_i} + u_{2^r+1} x_{2^r+1} = 0
\]
for all $n$, where not all $u_i$ vanish. We set $u_{I_i}:=u_i$ for ease of notation.

\par Suppose first that $u_{2^r+1}\neq 0$. Dividing by $\prod_{i=1}^r a_i^{r_n(i)+s_n(i)}$ and letting $n\to\infty$, we obtain
\[
u_{2^r+1}\alpha = u_{J},
\]
where $J=\emptyset$ and hence $\alpha\in \Q$.

\par Suppose now that $u_{2^r+1}=0$. Then
\[
\sum_{I} u_I x_I = 0,
\]
where $I$ ranges over the subsets of $\{1, \dots, r\}$. Dividing by $\prod_{i=1}^r a_i^{r_n(i)}$, we obtain a relation of the form
\[
\sum_{I} u_I \prod_{i\notin I} a_i^{s_n(i)} = 0.
\]
Let $I_0$ be such that $u_{I_0}\neq 0$. We find that 
\[\sum_{I\neq I_0} -\frac{u_I }{u_{I_0}}\frac{\prod_{i\notin I} a_i^{s_n(i)}}{\prod_{j\notin I_0} a_j^{s_n(j)}}=1.\]
Let $\widetilde{S}$ be a finite set of places containing $S$ such that $u_I/u_{I_0}$ are $\widetilde{S}$-units.
As a consequence of the generalized $S$-unit theorem \cite[Corollary 7.4.3]{BomGub} there is a finite set $\Phi$ such that for all $n$, there is a set $I\neq I_0$ such that 
\[\frac{\prod_{i\notin I} a_i^{s_n(i)}}{\prod_{j\notin I_0} a_j^{s_n(j)}}\in \Phi.\] Thus by the Pigeon hole principle, there exists a set $I'\neq I_0$ such that 
\[\frac{\prod_{i\notin I} a_i^{s_n(i)}}{\prod_{j\notin I_0} a_j^{s_n(j)}}\in \Phi\] for infinitely many $n$. Therefore there exists $n\neq m$ such that 
\[\frac{\prod_{i\notin I} a_i^{s_n(i)}}{\prod_{j\notin I_0} a_j^{s_n(j)}}=\frac{\prod_{i\notin I} a_i^{s_m(i)}}{\prod_{j\notin I_0} a_j^{s_m(j)}}.\]
Therefore
\[\sum_{i\in I} (s_n(i)-s_m(i))\log a_i-\sum_{j\in I_0} (s_n(j)-s_m(j))\log a_i=0.\] We write this relation as 
\[\sum_I w_i \log a_i=0,\] where 
\[w_i:=\begin{cases}
    s_n(i)-s_m(i) & \text{ if } i\in I\setminus I_0;\\
    s_m(i)-s_n(i) & \text{ if } i\in I_0\setminus I;\\
    0& \text{ otherwise}.
\end{cases}\]
Note that the sequence $(s_n(i))_n$ is strictly increasing and thus in particular, $s_n(i)\neq s_m(i)$ for all $i=1, \dots, r$. Since $I\neq I_0$, not all coefficients are $0$ and this gives us a contradiction since $\log a_1, \dots, \log a_r$ are $\Q$-linearly independent. Thus $\alpha\in \Q$ if $\alpha\in \bar{\Q}$ and this completes the proof.
\end{proof}

\bibliographystyle{alpha}
\bibliography{references}

\end{document}